\documentclass{amsart}

\usepackage{amsmath,amssymb,amsfonts,enumerate,amsthm,graphicx,color}
\def\opn#1#2{\def#1{\operatorname{#2}}} % to make operators
\opn\chara{char} \opn\length{\ell} \opn\pd{pd} \opn\rk{rk}
\opn\projdim{proj\,dim} \opn\injdim{inj\,dim} \opn\rank{rank}
\opn\depth{depth} \opn\grade{grade} \opn\height{height}
\opn\embdim{emb\,dim} \opn\codim{codim}

\opn\Tr{Tr} \opn\bigrank{big\,rank}
\opn\superheight{superheight}\opn\lcm{lcm}
\opn\trdeg{tr\,deg}%
\opn\reg{reg} \opn\lreg{lreg} \opn\skel{skel}
\opn\multideg{multideg}
\opn\div{div} \opn\Div{Div} \opn\cl{cl} \opn\Cl{Cl}
\opn\Spec{Spec} \opn\Supp{Supp} \opn\supp{supp} \opn\Sing{Sing}
\opn\Ass{Ass}
\opn\Ann{Ann} \opn\Rad{Rad} \opn\Soc{Soc}
\opn\Ker{Ker} \opn\Coker{Coker} \opn\Im{Im} \opn\Hom{Hom}
\opn\Tor{Tor} \opn\Ext{Ext} \opn\End{End} \opn\Aut{Aut}
\opn\id{id}

\opn\nat{nat}
\opn\pff{pf}%   \pf exists already
\opn\Pf{Pf} \opn\GL{GL} \opn\SL{SL} \opn\mod{mod} \opn\ord{ord}
\opn\aff{aff} \opn\con{conv} \opn\relint{relint} \opn\st{st}
\opn\lk{lk} \opn\cn{cn} \opn\core{core} \opn\vol{vol}
\opn\link{link} \opn\star{star} \opn\skel{skel}
\opn\gr{gr}

\def\pot#1#2{#1[\kern-0.28ex[#2]\kern-0.28ex]}

\opn\dirlim{\underrightarrow{\lim}}
\opn\inivlim{\underleftarrow{\lim}}
\def\Implies{\ifmmode\Longrightarrow \else
     \unskip${}\Longrightarrow{}$\ignorespaces\fi}
\def\implies{\ifmmode\Rightarrow \else
     \unskip${}\Rightarrow{}$\ignorespaces\fi}
\def\iff{\ifmmode\Longleftrightarrow \else
     \unskip${}\Longleftrightarrow{}$\ignorespaces\fi}

\let\:=\colon

\newtheorem{thm}{Theorem}[section]
\newtheorem{cor}[thm]{Corollary}
\newtheorem{lem}[thm]{Lemma}
\newtheorem{prop}[thm]{Proposition}

\newtheorem{exam}[thm]{Example}

\numberwithin{equation}{section}

\begin{document}
\bibliographystyle{amsplain}

%\date{}
\title{Resolution of Unmixed Bipartite Graphs }
\author{Fatemeh Mohammadi and Somayeh Moradi}
%%\thanks{}
\address{Fatemeh Mohammadi, Department of Pure Mathematics,
 Faculty of Mathematics and Computer Science,
 Amirkabir University of Technology (Tehran Polytechnic),
424, Hafez Ave., Tehran 15914, Iran.}
\email{f\_mohammadi@aut.ac.ir}
\address{Somayeh Moradi, Department of Pure Mathematics,
 Faculty of Mathematics and Computer Science,
 Amirkabir University of Technology (Tehran Polytechnic),
424, Hafez Ave., Tehran 15914, Iran.}
\email{s\_moradi@aut.ac.ir}
\begin{abstract}
For an unmixed bipartite graph $G$ we consider the lattice of vertex
covers $\mathcal{L}_G$ and compute depth, projective dimension and extremal Betti-numbers
of $R/I(G)$ in terms of this lattice.
\end{abstract}

\maketitle

\section*{Introduction}

In recent years edge ideals of graphs have been intensively studied.
One central question in this context is when the edge ideal is
Cohen-Macaulay. In the case of chordal graphs and bipartite graphs a
nice and complete answer is known, see \cite{HHZ1} and \cite{HHZ}. In particular, in these
cases one knows the projective dimension of the edge ideal. For
bipartite Cohen-Macaulay graphs the Cohen-Macaulay type is also
known. There are also several papers  in which the regularity of
edge ideals has been studied, (see \cite{HT}, \cite{HT2}, \cite{Z}). It is impossible to classify
all Cohen-Macaulay graphs or to give uniform formulas for the
projective dimension or the regularity for all graphs. Nevertheless
there are interesting classes of graphs where the homological data
of the edge ideals can be described. In this paper we consider
unmixed bipartite graphs and study the resolution of their edge
ideals.

To be more specific, let $G$ be a finite graph with vertex
set $V(G)=\{v_1,\ldots,v_n\}$  and edge set $E(G)$. Throughout we assume
that $G$ has no isolated vertices. The graph $G$ is called bipartite,
if $V(G)=X\cup Y$ with $X\cap Y= \emptyset$ such that $E(G)\subseteq  X\times Y$.

A vertex cover of $G$ is a subset $C$
of $V(G)$ such that each edge has at least one vertex in $C$. A
minimal vertex cover $C$ of $G$ is a vertex cover such that no
proper subset of $C$ is a vertex cover of $G$. The graph $G$ is
called unmixed if all its minimal vertex covers are of the same
cardinality.

Let $G$ be an unmixed bipartite graph with vertex set
 $V(G)=\{x_1,\ldots,x_n\}\cup \{y_1,\ldots,y_m\}$. Since $G$ is unmixed, it follows
that $n=m$. For any unmixed bipartite graph there is a perfect
matching, see \cite{V}. Therefore we may  assume that $\{x_i,y_i\}$ is an
edge of $G$ for all $i$. So each minimal vertex cover of $G$ is of
the form $\{x_{i_1},\ldots,x_{i_s},y_{i_{s+1}},\ldots,y_{i_n}\}$,
where $\{i_1,\ldots,i_n\}=[n]$. For any minimal vertex cover $C$ of
$G$ let $\bar{C}=C\cap \{x_1,\ldots,x_n\}$. Let $B_n$ be the Boolean
lattice on the set $\{x_1,\ldots,x_n\}$. In \cite{HHO} it is shown
that for any unmixed bipartite graph $G$, the subset
$$\mathcal{L}_G=\{\bar{C}: C {\rm \ is\ a\ minimal\ vertex\ cover\ of \
}G\}\subseteq \{x_1,\ldots,x_n\}$$ is a sublattice of $B_n$ which contains $\emptyset $ and $X$, and for any such
sublattice $\mathcal{L}$ of $B_n$, there exists an unmixed bipartite graph $G$ such that $\mathcal{L}=\mathcal{L}_G$. In our
description of the resolution of the edge ideal of an unmixed bipartite
graph, we use in a substantial way this characterization of such graphs  in terms of their vertex cover lattices.

In the following let $\mathcal{L}_G$ be a sublattice of the Boolean
lattice $B_n$ corresponding to the unmixed bipartite graph $G$.
Therefore $\wedge$ and $\vee$ in $\mathcal{L}_G$ are just taking the
intersection and union. Attached to the lattice $\mathcal{L}_G$ is a
monomial ideal $H_{\mathcal{L}_G}$ in the polynomial ring
$K[x_1,\ldots,x_n,y_1,\ldots,y_n]$, defined as follows: for each
element $p\in \mathcal{L}_G$, let $u_p=X_p Y_{[n]\setminus p}$,
where $X_p=\prod_{i\in p}x_i$ and $ Y_{[n]\setminus p}=\prod_{j\in
[n]\setminus p}y_j$. Then $H_{\mathcal{L}_G}$ is generated by the
monomials $u_p$, $p\in \mathcal{L}_G$. The edge ideal $I(G)$ which
we are interested in is just the Alexander dual of
$H_{\mathcal{L}_G}$. Thus we may apply the
Bayer-Charalambous-Popescu theory \cite[Theorem~2.8]{BCP} which
relates the multigraded extremal Betti-numbers of
$H_{\mathcal{L}_G}$ and $I(G)$.

In Section 1 we describe the multigraded minimal free resolution of
the monomial ideal $H_{\mathcal{L}_G}$.  In fact, the resolution we
describe in Theorem~\ref{res} is a variation of the  resolution of
Hibi ideals of meet-distributive semilattices,  given in
\cite[Theorem~2.1]{HHZ}. In Theorem~\ref{dif} we also describe the
differentials of this resolution. This information is not needed
later. More important is the fact, that the multigraded basis
elements of the resolution, can be identified with the Boolean
sublattices of $\mathcal{L}_G$, see Proposition~\ref{3}. Having this
identification, it turns out that the multigraded extremal
Betti-numbers correspond to the maximal Boolean sublattices of
$\mathcal{L}_G$. In Section 2 we use Alexander duality and Theorem
in \cite{BCP} to obtain the multigraded extremal Betti-numbers of
$I(G)$. With this information at hand, we can express the depth and
regularity of $R/I(G)$ in terms of the lattice $\mathcal{L}_G$, see
Corollary~\ref{cor1}. As a further corollary we obtain a lower bound
for the last nonzero total Betti-number of $R/I(G)$. We do not know
of any example where this lower is not achieved. It would be always
achieved if one could prove the following: all nonzero multigraded
Betti-number in the last step of the resolution are extremal (in the
multigraded sense). There is a simple argument, given in the proof
of Proposition \ref{cm} that whenever $I$ is monomial Cohen-Macaulay
ideal, then all multigraded extremal Betti-numbers appear  at the
end of the resolution of $I$.

After having finished the paper,  Professor Herzog informed us
that Kummini  in his thesis and in a preprint paper has also computed the depth and regularity of unmixed bipartite graphs, see \cite{K2}. His approach and the terms in which he expresses these invariants differ from ours.

\section{Minimal Free Resolution of $\mathcal{L}_G$}

The purpose of this section it to construct a resolution of the
ideal $H_{\mathcal{L}_G}$ which is a modification of the resolution
given \cite[Theorem~2.1]{HHZ} adopted to our situation. The
differences between our resolution and the one given in \cite{HHZ}
arises from the fact that the lattices under consideration are
differently embedded into Boolean lattice. This fact is important,
because the multidegrees of the resolution depends on the embedding.

In order to guarantee the minimality  of the resolution which we are going to describe we need the following result.

\begin{lem}\label{lem}
Let $p\in \mathcal{L}_G$. For any two distinct subsets
$S,S'\subseteq N(p)$, we have $\wedge\{q; q\in S\}\neq \wedge\{q;
q\in S'\}$.
\end{lem}

\begin{proof}
{Let $S,S'\subseteq N(p)$ be distinct subsets. One can assume that
$S\nsubseteq S'$. Let $q_1\in S\setminus S'$. If $\wedge\{q; q\in
S\}=\wedge\{q; q\in S'\}$, then $(\wedge\{q; q\in S\})\vee q_1=
(\wedge\{q; q\in S'\})\vee q_1$. Since $\mathcal{L}_G$ is
distributive $\wedge(\{q\vee q_1; q\in S\})= \wedge(\{q\vee q_1;
q\in S'\})$. For any $q\in S'$, $q\vee q_1=p$, therefore
$\wedge(\{q\vee q_1; q\in S'\})=p$. But then $\wedge(\{q\vee q_1;
q\in S\})=q_1$, a contradiction. }
\end{proof}

From the above lemma it is easy to see that for any subsets $S$ and
$S'$ of $N(p)$ with $S\subseteq S'$, we have $|S'|-|S|\leq
|\wedge\{q; q\in S\}|-|\wedge\{q; q\in S'\}|$. In the following we denote by $\hat{0}$ and $\hat{1}$
the minimal and maximal element of $\mathcal{L}_G$. For any $p\in
\mathcal{L}$, rank of  $p$ which is denoted by rank$(p)$,
 is the maximal length of chains descending from $p$.
We extend the partial rank order on
 $\mathcal{L}_G$ to a total order $\prec$.

\begin{thm}\label{res}
There exists a minimal multigraded free resolution $\mathbb{F}$ of
$H_{\mathcal{L}_G}$ such that for each $i\geq 0$, the free module
$\mathbb{F}_i$ has a basis with basis elements $b(p;S)$, where $p\in
\mathcal{L}_G$ and $S$ is a subset of the set of lower neighbors
$N(p)$ of $p$ with $|S|=i$ and multidegree of $b(p;S)$ is the least
common multiple of $u_p$ and all $u_q$ with $q\in S$.
\end{thm}

\begin{proof}
{The construction of resolution is as in the proof of
\cite[Theorem~2.1]{HHZ} by mapping cone. For any $p\in
\mathcal{L}_G$ we construct inductively a complex $\mathbb{F}(p)$
which is a multigraded free resolution of the ideal
$H_{\mathcal{L}_G}(p)=\langle u_q ; q\preceq p\rangle$. The complex
$\mathbb{F}(\hat{0})$ is defined as $\mathbb{F}_i(\hat{0})=0$ for
$i>0$ and $\mathbb{F}_0(\hat{0})=S$. Now, let $p\in \mathcal{L}_G$
and $q\in \mathcal{L}_G$, $q\prec p$ be the element preceding $p$.
Then $H_{\mathcal{L}_G}(p)=(H_{\mathcal{L}_G}(q),u_p)$, and hence we
have the exact sequence of multigraded $S$-modules
$$0\longrightarrow (S/L)(-\multideg u_p)\longrightarrow
S/H_{\mathcal{L}_G}(q)\longrightarrow
S/H_{\mathcal{L}_G}(p)\longrightarrow 0,$$ where $L$ is the colon
ideal $H_{\mathcal{L}_G}(q):u_p$. Let $u_{q'}/[u_{q'},u_p]\in L$,
where $q'<q$ and let $t\in [q'\wedge q,p]\cap N(p)$. Then
$u_t/[u_t,u_p] $ divides $u_{q'\wedge p}/[u_{q'\wedge p},u_p]$ and
$u_{q'\wedge p}/[u_{q'\wedge p},u_p]$ divides $ u_q'/[u_q',u_p]$.
Therefore we have $L=(\{u_t/[u_t,u_p]\}_{t\in N(p)})$.

Let $\mathbb{T}$ be the Taylor complex associated with the sequence
$u_t/[u_t,u_p]$, $t\in N(p)$, where the order of the sequence is
given by the order $\prec$ on the elements of $\mathbb{L}$. Then
${T}_i$ has a basis with elements $e_{t_1}\wedge e_{t_2}\wedge\cdots
\wedge e_{t_i}$, where $t_1\prec t_2\prec\cdots\prec t_i$. The
multidegree of $e_{t_1}\wedge e_{t_2}\wedge\cdots \wedge e_{t_i}$ is
the least common multiple of the elements $u_{t_j}/[u_{t_j},u_p]$
for $j=1,\ldots,i$. The shifted complex $\mathbb{T}(-\multideg
(u_p))$ is a multigraded free resolution of $(S/L)(-\multideg
(u_p))$. Let $b(p;t_1,\ldots,t_i)$ be basis element of
$\mathbb{T}_i(-\multideg (u_p))$ corresponding to $e_{t_1}\wedge
e_{t_2}\wedge\cdots \wedge e_{t_i}$. Then
multideg$(b(p;t_1,\ldots,t_i))$=multideg$(u_p)+$multideg
$(e_{t_1}\wedge e_{t_2}\wedge\cdots \wedge
e_{t_i})=lcm(u_p,u_{t_1},\\\ldots,u_{t_i})$. This resolution is
minimal since for any $t_1\prec t_2\prec \cdots\prec t_i$ we have
$lcm(f_1,\ldots,f_i)/lcm(f_1,\ldots,\check{f}_j,\ldots,f_i)=Y_{p\setminus
\bigwedge_{l=1}^i t_l}/Y_{p\setminus \bigwedge_{l=1,l\neq j}^i
t_l}\neq 1$.

The monomorphism $(S/L)(-\multideg (u_p))\longrightarrow
S/H_{\mathcal{L}_G}(q)$ induces a comparison map
$\alpha:\mathbb{T}(-\multideg (u_p))\longrightarrow \mathbb{F}(q)$
of multigraded complexes. Let $\mathbb{F}(p)$ be the mapping cone of
$\alpha$. Then $\mathbb{F}(p)$ is a multigraded free $S$-resolution
of $H_{\mathcal{L}_G}(p)$ with the desired multigraded basis.

We claim that this resolution is minimal. For any two basis elements
$b(p;S)$ and $b(q;T)$ with $|T|=|S|-1$. We show that the coefficient
of $b(q;T)$ in $\partial b(p;S)$ is either zero or a monomial $\neq
1$. First assume that $p=q$. If $T\subseteq S$, then  the
coefficient is multideg$(b(p;S))/$ multideg$(b(p,T))=Y_A$, where
$A=\wedge\{r; r\in T\}\setminus \wedge\{r; r\in S\}$. Since $A$ is a
nonempty set by Lemma~\ref{lem}, then $Y_A\neq 1$. If $T\nsubseteq S$ and
multideg$(b(p;T))$ divides multideg$(b(p;S))$, then $\wedge\{r; r\in
S\}\leq \wedge\{r; r\in T\}$.  Therefore $\wedge\{r; r\in
N(p)\}=\wedge\{r; r\in N(p)\setminus (T\setminus S)\}$, which is a
contradiction by Lemma\ref{lem}. Now, assume that $q<p$. If  multideg$(b(q;T))$ divides
multideg$(b(p;S))$, then the coefficient is $X_{p\setminus q}Y_B$ for
some set $B\subseteq [n]$ and so it is not $1$. In the remaining
case $q\nless p$, multideg$(b(q;T))$ does not divide multideg$(b(p;S))$. }
\end{proof}

In the following theorem we describe the maps in the resolution of Theorem~\ref{res}.
These information is just for curiosity and is not needed to get the results which come after.
Before describing the maps in the resolution we fix some notation.
For a subset $S\subset \mathcal{L}_G$ and $q\in S$ let
$\sigma(q;S)=|\{r\in S; r\prec q\}|$. Let $X_{p\setminus
q}=\prod_{i\in p\setminus q}x_i$ and $Y_{S_q}=\prod_{i\in
(\bigwedge_{r\in S\setminus \{q\}} r)\setminus q}y_i$.

\begin{thm}\label{dif}
For any $p\in \mathcal{L}_G$ and $S\subseteq N(p)$ we have
$$\partial(b(p;S))=\sum_{q\in S}(-1)^{\sigma(q;S)}(Y_{S_q}b(p;S\setminus \{q\})-
X_{p\setminus q}b(q;q\wedge (S\setminus \{q\}))).$$
 \end{thm}
\begin{proof}{
First we remind that for any $q_i$ and $q_j$ in $N(p)$ we have $q_i\wedge q_j\in
N(q_i)$, since $[q_i\wedge q_j,p]$ is isomorphic to a Boolean
lattice of rank two. Then we have $q\wedge (S\setminus \{q\})\subseteq N(q)$.
We have
multideg$(b(p;S))=X_pY_{[n]\setminus (\bigwedge_{r\in S} r)}$ and
multideg$(b(p;S\setminus \{q\}))=X_pY_{[n]\setminus (\bigwedge_{r\in
S\setminus q} r)}$. Also,
$\{i;i\notin \bigwedge_{r\in S} r\}\setminus \{ i;i\notin \bigwedge_{r\in S\setminus q} r\}=
\{i;i\in \bigwedge_{r\in S\setminus q} r\setminus q\}$ shows that
$Y_{[n]\setminus (\bigwedge_{r\in S}
r)}/Y_{[n]\setminus (\bigwedge_{r\in S\setminus q} r)}=Y_{S_q}$ and so multideg$(b(q;q\wedge (S\setminus \{q\})))=X_qY_{[n]\setminus \bigwedge_{r\in S\setminus q}
(q\wedge r)}=X_qY_{[n]\setminus \bigwedge_{r\in S}
(r)}=$multideg$(b(p;S))/X_{p\setminus q}$.
Therefore, $\partial$ is a multi-homogeneous differential.
Consider the mapping cone constructed in Theorem~\ref{res}. The differential given by
that mapping cone is
$$\partial_i=(\partial_i^{\mathbb{T}}+(-1)^i \alpha_i,\partial_{i+1}^{\mathbb{F}(q)})
{\rm \ for\  all\ } i.$$

Comparing this equation with $\partial$ which
is defined in theorem it is enough to show that for any $S\subseteq N(p)$ we have:

\begin{itemize}

\item[(i)] $ \partial^{\mathbb{T}}(b(p;S))=
\sum_{q\in  S}(-1)^{\sigma(q,S)}Y_{S_q} b(p;S\setminus\{q\}),$
which is exactly the definition of Taylor complex.

\item[(ii)]
 We can choose $\alpha$ such that
$$(-1)^i \alpha_i(b(p;S))=-\sum_{q\in S}(-1)^{\sigma(q,S)}X_{p\setminus q} b(q;q\wedge S\setminus\{q\}).$$
\end{itemize}

We show that $\alpha:\mathbb{C}\rightarrow \mathbb{F}(q)$ is a complex homomorphism. So it
is enough to show that for any $b(p;S)\in T_i$
$$\partial_i^{F(q)}\alpha_i(b(p;S))=
\alpha_{i-1}\partial_i^{\mathbb{T}}(b(p;S)).$$
First we compute the left side of the above equation. Then we have
\\
\\
$(1)\ \ \partial_i^{F(q)}\alpha_i(b(p;s))=(-1)^{i+1}\sum_{q\in S}(-1)^{\sigma(q;S)}X_{p\setminus q}
\partial_i^{F(q)}(b(q;q\wedge (S\setminus \{q\}))).$
\\
\\
By induction hypothesis we have:
\\
\begin{eqnarray*} (2)\ \ \partial_i^{F(q)}(b(q;q\wedge (S\setminus \{q\})))
&=&\sum_{q'\in S\setminus q}(-1)^
{\sigma(q';S\setminus\{q\})}Y_{(q\wedge S)_{q\wedge q'}} b(q;q\wedge (S\setminus \{q,q'\}))
\\\\
&-&X_{q\setminus (q\wedge q')} b(q\wedge q';q'\wedge [(q\wedge (S\setminus \{q\})
\setminus \{q,q'\}])).
\end{eqnarray*}
\\
Since $q\vee q'=p$ for any $q,q'\in N(p)$ we have $q\setminus (q\wedge q')=p\setminus q'$.
Also from
$\bigwedge_{r\in S\setminus\{q,q'\}}\setminus \{q,q'\}=
\bigwedge_{r\in S\setminus\{q'\}}\setminus \{q,q'\}=
\bigwedge_{r\in S\setminus\{q'\}}\setminus \{q'\}$ we see that
$Y_{(q\wedge S)_{q\wedge q'}}=Y_{S_{q'}}$.
Considering these statements and substituting $(2)$ in $(1)$ we get
\\
\\
$
(3)\ \ \partial_i^{F(q)}\alpha_i(b(p;s))=\\\\(-1)^{i+1}\sum_{q,q'\in S,q\neq q'}
(-1)^{\sigma(q;S)+\sigma(q';S\setminus \{q\})} X_{p\setminus q} Y_{S_{q'}}
b(q;q\wedge (S\setminus \{q,q'\})).
$
\\
\\
On the other hand we have
\\
\\
$
(4)\ \ \alpha_{i-1}\partial_i^{\mathbb{T}}(b(p;S))=(-1)^{\sigma(q;S)} Y_{S_{q}}\alpha_{i-1}
(b(p;S\setminus\{q\}))$
\\

After substituting $ \alpha_{i-1}(b(p;S\setminus\{q\})) $ in $(4)$ and using
$q\setminus (q\wedge q')=p\setminus q'$ we get
\\
\\
$
(5)\ \ \alpha_{i-1}\partial_i^{\mathbb{T}}(b(p;S))=\\\\(-1)^{i+1}\sum_{q,q'\in S,q\neq q'}
(-1)^{\sigma(q;S)+\sigma(q';S\setminus \{q\})}Y_{S_{q}} X_{p\setminus q'}
b(q';q'\wedge (S\setminus \{q,q'\})).
$
\\

After exchanging $q$ and $q'$ we see that the equations $(3)$ and
$(5)$ are equal, which completes the proof. }
\end{proof}

The next observation is of crucial importance for understanding the
$i$-extremal and extremal Betti-numbers of $H_{\mathcal{L}_G}$ and
$I(G)$.

\begin{prop}\label{3}
There is a correspondence between the basis elements
$b(p;S)$ and intervals in $\mathcal{L}_G$, which are isomorphic to a
Boolean lattice.

\end{prop}

\begin{proof}
{For any basis element $b(p;S)$, we show that the interval
$[\wedge\{q; q\in S\},p]$ is isomorphic to $B_{|S|}$. Let
$S=\{q_1,\ldots,q_n\}$ and $v_i=\wedge\{q; q\in S\setminus
\{q_i\}\}$. The interval $[\wedge\{q; q\in S\},p]$ is a Boolean
lattice on $v_1,\ldots,v_n$.

For any $I\in [\wedge\{q; q\in S\},p]$, let $i_1,\ldots,i_k$ be the
indices that $I\leq q_{i_j}$, $1\leq j\leq k$. Then $I\leq
\wedge\{q_{i_j}; 1\leq j\leq k\}$. If $I\neq \wedge\{q_{i_j}; 1\leq
j\leq k\}$, then there exists $x\in \wedge\{q_{i_j}; 1\leq j\leq
k\}\setminus I$. Since $\wedge\{q; q\in S\}\leq I$, there exists
$1\leq l\leq n$, $l\neq i_1,\ldots,i_k$ such that $x\notin q_l$. But
then $x\notin I\vee q_l=p$, a contradiction. Thus $I=
\wedge\{q_{i_j}; 1\leq j\leq k\}=\vee \{v_j; 1\leq j\leq n, j\neq
i_1,\ldots,i_k\}$. Let $\Phi$ be a function from the set of basis
elements to intervals in $\mathcal{L}_G$, which are isomorphic to a
Boolean lattice such that $\Phi(b(p;S))=[\wedge\{q; q\in S\},p]$.
From Lemma~\ref{lem} we know that $\Phi$ is a monomorphism. For any
interval $[J,I]$ in $\mathcal{L}_G$ isomorphic to a Boolean lattice,
set $S=N(I)\wedge [J,I]$. Then $[J,I]=[\wedge\{q; q\in S\},I]$ and
$\Phi$ is surjective. } \end{proof}

In the following we denote by  $A_G$ the set of  elements $p\in \mathcal{L}_G$ such that the
interval $[\wedge\{r; r\in N(p)\},p]$ is isomorphic to a maximal
Boolean lattice in $\mathcal{L}_G$.

\begin{lem}\label{lem2}
Let $p\in A_G$. For any $q\in \mathcal{L}_G$ such that $[\wedge\{r;
r\in N(q)\},q]\subseteq [\wedge\{r; r\in N(p)\},p]$, we have
$$|q|-|N(q)|-|\wedge\{r; r\in N(q)\}|\leq |p|-|N(p)|-|\wedge\{r;
r\in N(p)\}|.$$
\end{lem}

\begin{proof}
Let $[\wedge\{r; r\in N(p)\},p]$ be a Boolean lattice on the
elements $v_1,\ldots,v_n$. Then $|N(p)|=n$, $|N(q)|\leq n$ and
$v_i\wedge v_j=\wedge\{r; r\in N(p)\}$ for any $1\leq i<j\leq n$.
Also $v_1\vee \cdots \vee v_n=p$. Without loss of generality assume
that $q=v_1\vee \cdots \vee v_m$ for some $m<n$. Then $|N(q)|=m$. We
claim that for any $1\leq i\leq n$, there exists an element $x_i\in
v_i$ such that $x_i$ is not in any other $v_j$, $1\leq j\leq n$.
Otherwise let $1\leq i\leq n$ be such that for any $x\in v_i$, there
exists $j\neq i$ with $x\in v_j$. It means that $v_i\leq \wedge\{r;
r\in N(p)\}$, a contradiction. Thus $|p|-|q|\geq n-m$. Since
$|\wedge\{r; r\in N(p)\}|\leq |\wedge\{r; r\in N(q)\}|$, we get the
inequality.
\end{proof}

As a first corollary we obtain

\begin{cor}\label{cor1}
Let $G$ be an unmixed bipartite graph on $(X,Y)$ such that
$|X|=|Y|=n$ and $A_G\subseteq \mathcal{L}_G$ be the set defined
above. Then
\begin{enumerate}
\item[{\em (i)}]  $\depth(R/I(G))=n-\max_{p\in A_G}\{|p|-|N(p)|-|\wedge\{r; r\in
N(p)\}|\}$;

\item[{\em (ii)}] $\reg(R/I(G))=\max_{p\in \mathcal{L}_G}|N(p)|.$
\end{enumerate}

\end{cor}

\begin{proof}
{For any basis element $b(p;S)$ in $\mathbb{F}_{|S|}$, one has
$\multideg(b(p;S))=X_pY_{[n]\setminus \wedge\{r; r\in S\}}$ and so
deg$(b(p;S))=|p|+n-|\wedge\{r; r\in S\}|$. For $S\subseteq N(p)$
from Lemma \ref{lem} we have $|\wedge\{r; r\in N(p)\}|+|N(p)|\leq
|\wedge\{r; r\in S)\}|+|S|$. Therefore by Lemma \ref{lem2} $${\rm
reg\ }(H_{\mathcal{L}_G})=n+\text{max}_{p\in
A_G}\{|p|-|N(p)|-|\wedge\{r; r\in N(p)\}|\}.$$ Since
$\reg(H_{\mathcal{L}_G})=\pd(R/I(G))$, from Auslander-Buchsbaum
formula one has $(i)$. The other statement is clear from
Theorem~\ref{res}. }
\end{proof}

Recall that a multigraded Betti-number $\beta_{i,b}$ is called
$i$-extremal if $\beta_{i,c}=0$ for all $c>b$, that is all
multigraded entries below $b$ on the $i$-th column vanish in the
Betti diagram as a Macaulay output [Mac]. Define $\beta_{i,b}$ to be
extremal if $\beta_{j,c}=0$ for all $j\geq i$, and $c>b$ so
$|c|-|b|\geq j-i$. In other words, $\beta_{i,b}$ corresponds to the
"top left corner" of a box of zeroes in the multigraded Betti
diagram. A graded Betti-number $\beta_{i,r}$ is called extremal if
$\beta_{j,l}=0$ for any $j\geq i$, $l>r$ and $l-r\geq j-i$. In other
words, $\beta_{i,r}$ corresponds to the "top left corner" of a box
of zeroes in the multigraded Betti diagram.

\begin{cor}
All multigraded Betti-numbers in homological degree $i$  are
$i$-extremal. A multigraded Betti-number $\beta_{i,b}$ of
$H_{\mathcal{L}_G}$ is extremal if and only if there exists $p\in
A_G$ such that $i=|N(p)|$ and $b=\multideg(b(p;N(p)))$.
\end{cor}

\begin{proof}
Let $b(p;S)$ and $b(p';S')$ be two basis elements of $\mathbb{F}_i$
such that $\multideg(b(p;S))$ divides $\multideg(b(p';S'))$. Then
$p\leq p'$ and $\wedge\{q; q\in S'\}\leq \wedge\{q; q\in S\}$ and
$|S|=|S'|=i$. Therefore  $[\wedge\{q; q\in S\},p]\subseteq
[\wedge\{q; q\in S'\},p']$. Since both intervals are isomorphic to a
Boolean lattice of rank $i$, one has $p=p'$ and $\wedge\{q; q\in
S\}=\wedge\{q; q\in S'\}$. Since $S,S'\subseteq N(p)$ from Lemma
\ref{lem} we have $S=S'$. Therefore all multigraded Betti-numbers in
homological degree $i$ are $i$-extremal.

For any basis element $b(p;S)$, $\multideg(b(p;S))=X_pY_{[n]\setminus
\wedge\{q; q\in S\}}$.  Therefore $\multideg(b(p;S))$ divides
$\multideg(b(p;N(p)))$ and $|N(p)|-|S|\leq |\wedge\{q; q\in
S\}|-|\wedge\{q; q\in N(p)\}|$. Therefore $\beta_{i,b}$ is extremal
only if $b=\multideg(p;N(p))$ for some $p\in \mathcal{L}_G$. Let $p$
and $q$ be two elements of $\mathcal{L}_G$ and $\multideg(b(p;N(p)))$
divides $\multideg(b(q;N(q)))$. Then $p\leq q$ and $\wedge\{r; r\in
N(q)\}\leq \wedge\{r; r\in N(p)\}$. Therefore $[\wedge\{r; r\in
N(p)\},p]\subseteq [\wedge\{r; r\in N(q)\},q]$. Then
$\beta_{|N(p)|,b}$ is extremal precisely when
$b=\multideg(b(p;N(p)))$, where $[\wedge\{q; q\in N(p)\},p]$ is a
maximal Boolean sublattice of $\mathcal{L}_G$.
\end{proof}

The graded extremal Betti-numbers of $R/I(G)$ can be seen from the
lattice $\mathcal{L}_G$. Indeed, the Betti-number $\beta_{i,i+j}$
for $R/I(G)$ is extremal if and only if there exists $p\in A_G$ with
$|N(p)|=j$ and $n+|p|-|N(p)|-|\wedge\{r; r\in N(p)\}|=i$ such that:
\begin{enumerate}
\item[(a)] For any $q\in A_G$ with $|N(q)|>|N(p)|$, one has
$$|q|-|N(q)|-|\wedge\{r; r\in N(q)\}|<|p|-|N(p)|-|\wedge\{r; r\in
N(p)\}|.$$

\item[(b)] For any $q\in A_G$ with $|N(q)|=|N(p)|$, one has
$$|q|-|\wedge\{r; r\in N(q)\}|\leq |p|-|\wedge\{r; r\in N(p)\}|.$$
\end{enumerate}

This statement follows immediately from the definition of graded extremal Betti-numbers by using \cite[Theorem~2.8]{BCP}.

\begin{exam}
{\em For any $p\in \mathcal{L}_G$ let $f(p)=|p|-|N(p)|-|\wedge\{r;
r\in N(p)\}|$. In the following lattice we have $f(p)=1$ for
$p=\{1,2,3\}$, $p=\{1,2,3,4\}$ and $p=\{1,2,3,4,5,6,7\}$ and
$f(p)=0$ for $p=\{1,2,3,5\}$ and $p=\{1,2,3,4,5\}$. Therefore by the
statement above, the extremal Betti-number of the unmixed bipartite
graph corresponding to this lattice is $\beta_{i,i+j}$ for
$i=7+f(\{1,2,3,4\})=7+f(\{1,2,3\})=8$ and $j=2$ and
$\beta_{8,10}=2$.}
\end{exam}
\vskip1.5truecm

\begin{center}{
\setlength{\unitlength}{2600sp}%
\begingroup\makeatletter\ifx\SetFigFont\undefined%
\gdef\SetFigFont#1#2#3#4#5{%
  \reset@font\fontsize{#1}{#2pt}%
  \fontfamily{#3}\fontseries{#4}\fontshape{#5}%
  \selectfont}%
\fi\endgroup%
\begin{picture}(2446,2420)(196,-3927)
{\color[rgb]{0,0,0}\thinlines \put(1419,-1548){\circle*{78}}
}%
{\put(1419,-1548){\line( 0,1){900}}
}%
{\put(1419,-670){\circle*{78}}
}%
{\put(1000,-610){1234567}
}%
{\put(1200,-1210){1}
}%

{\color[rgb]{0,0,0}\thinlines \put(1470,-1448){12345}
}%
{\color[rgb]{0,0,0}\put(1996,-2122){\circle*{78}}
}%
{\color[rgb]{0,0,0}\put(2040,-2122){1234}
}%
{\color[rgb]{0,0,0}\put(1320,-2122){0}
}%

{\color[rgb]{0,0,0}\put(2000,-2722){1}
}%

{\color[rgb]{0,0,0}\put(2600,-2730){\circle*{78}}
}%
{\color[rgb]{0,0,0}\put(2660,-2705){34}
}%
{\color[rgb]{0,0,0}\put(2010,-3319){\circle*{78}}
}%
{\color[rgb]{0,0,0}\put(2120,-3359){1}
}%
{\color[rgb]{0,0,0}\put(1420,-3359){1}
}%

{\color[rgb]{0,0,0}\put(1419,-3886){\circle*{78}}
}%
{\color[rgb]{0,0,0}\put(1350,-4140){$\emptyset$}
}%
{\color[rgb]{0,0,0}\put(1000,-4500){Figure 1}
}%
{\color[rgb]{0,0,0}\put(1419,-3886){\circle*{78}}
}%
{\color[rgb]{0,0,0}\put(829,-3319){\circle*{78}}
}%
{\color[rgb]{0,0,0}\put(430,-3340){12}
}%
{\color[rgb]{0,0,0}\put(238,-2729){\circle*{78}}
}%
{\color[rgb]{0,0,0}\put(-200,-2729){125}
}%
{\color[rgb]{0,0,0}\put(819,-2138){\circle*{78}}
}%
{\color[rgb]{0,0,0}\put(250,-2138){1235}
}%
{\color[rgb]{0,0,0}\put(1396,-2729){\circle*{78}}
}%
{\color[rgb]{0,0,0}\put(1196,-2529){123}
}%
{\color[rgb]{0,0,0}\put(796,-2729){0}
}%

{\color[rgb]{0,0,0}\put(1418,-1547){\line( 1,-1){1157}}
}%
{\color[rgb]{0,0,0}\put(248,-2739){\line( 1,-1){1157}}
}%
{\color[rgb]{0,0,0}\put(2599,-2739){\line(-1,-1){1157}}
}%
{\color[rgb]{0,0,0}\put(1396,-1547){\line(-1,-1){1157}}
}%
{\color[rgb]{0,0,0}\put(838,-2149){\line( 1,-1){1157}}
}%
{\color[rgb]{0,0,0}\put(1980,-2149){\line( -1,-1){1157}}
}%
\end{picture}%

}\end{center}
\vskip1.5truecm

As a final application we give a lower bound for the last nonzero
total Betti-number of an unmixed bipartite graph. To describe the
result, we introduce the set $B_G\subseteq A_G$ consisting
 of all elements $q$ such that $|q|-|N(q)|-|\wedge\{r; r\in
N(q)\}|=\max_{p\in A_G}\{|p|-|N(p)|-|\wedge\{r; r\in N(p)\}|\}$.

For an $R$-module $M$, let $t(M)$ denotes the last nonzero total
Betti-number of $M$. Then we have the following corollary.

\begin{cor}\label{total}
Let $G$ be an unmixed bipartite graph. Then $t(R/I(G))\geq |B_G|$.
\end{cor}

\begin{proof}
Let $p\in B_G$ and $r=\pd(R/I(G))$. Then $r=n+|p|-|N(p)|-|\wedge\{r;
r\in N(p)\}|$. Let $b=b(p;N(p))$, from \cite[Theorem~2.8]{BCP} we
have $\beta_{r,b}(R/I(G))=\beta_{|N(p)|,b}(H_{\mathcal{L}_G})=1$.
Therefore $T(R/I(G))\geq|B_G|$.
\end{proof}

We do not know of any example of a monomial ideal $I$ of projective
dimension $r$ for which there exists a nonzero multigraded
Betti-number $\beta_{r,b}$ which is not extremal. If such ideals
don't exist, at least among the edge ideals of unmixed bipartite
graphs, then we would have equality in Proposition \ref{total}.

In general the  multigraded extremal Betti-numbers of monomial ideal not only appear in the last step of the resolution. However we have

\begin{prop}\label{cm}
Let $R=k[x_1,\ldots,x_n]$ for a field $k$ and $I$ a monomial ideal
of $R$ such that $R/I$ is a Cohen-Macaulay ring. Then for all multigraded extremal Betti-numbers $\beta_{i,b}(R/I)$  we have $i=\pd(R/I)$.
\end{prop}

\begin{proof}
{ Assume that $$0\rightarrow F_r\rightarrow \cdots \rightarrow
F_{i+1}\rightarrow F_i\rightarrow \cdots \rightarrow F_0\rightarrow
R/I\rightarrow 0$$ is a minimal graded free resolution of $R/I$,
where $r=\pd(R/I)$ and $\varphi$ be the function $F_{i+1}\rightarrow
F_i$ in the resolution. Let $\beta_{i,b}(R/I)$ be a multigraded
extremal Betti-number with $i<\pd(R/I)$ and $e$ be the basis element
in $F_i$ with multidegree $b$. Then for any basis element $g$ of
$F_{i+1}$ the coefficient of $e$ in $\varphi(g)$ is zero. Otherwise
$\multideg(e)<\multideg(g)$ and $g\in F_{i+1}$ but then
$\beta_{i,b}$ is not extremal, a contradiction. This means that
$e^*$ is a cycle in the resolution $$0\rightarrow (R/I)^*\rightarrow
F_0^*\rightarrow \cdots \rightarrow F_i^*\rightarrow
F_{i+1}^*\rightarrow \cdots \rightarrow F_r^*\rightarrow 0,$$ where
$F_i^*=\Hom_R(F_i,R)$. Therefore $\Ext^i_R(R/I,R)\neq 0$. But we
know that $\Ext^j_R(R/I,R)=0$ for any $j<r$, a contradiction.
 }
\end{proof}

\providecommand{\byame}{\leavevmode\hbox
to3em{\hrulefill}\thinspace}


\begin{thebibliography}{10}
\bibitem{BCP} D.~Bayer, H.~Charalambous, S.~Popescu, {\em Extremal
Betti numbers and applications to monomial ideals,}  J. Algebra  221  (1999),  no. 2, 497--512.

\bibitem{HT} H.~T.~H$\grave{a}$, A.~Van Tyul, {\em Resolutions of squarefree monomial ideals via facet ideals: a survey. }
(2006) arXiv preprint
math.AC/0604301.

\bibitem{HT2} H.~T.~H$\grave{a}$, A.~Van Tyul, {\em Monomial ideals, edge ideals of hypergraphs, and their graded Betti numbers,} J. Algebraic Combin. 27 (2008), no. 2, 215–245.

\bibitem{HH1} J.~Herzog, T.~Hibi, {\em
  Distributive lattices, bipartite graphs and Alexander duality,} J. Alg. Comb. 22 (2005),
289--302.

\bibitem{HHO} J.~Herzog, T.~Hibi, H.~Ohsugi, {\em
  Unmixed bipartite graphs and sublattices of the Boolean lattices,} arXiv preprint math.AC/0806.1088v1 (2008).

\bibitem{HHZ1}
J.~Herzog, T.~Hibi, X.~Zheng,
{\em Cohen-Macaulay chordal graphs,} (English summary)
J. Combin. Theory Ser. A 113 (2006), no. 5, 911--916.
13H10 (05C99)

\bibitem{HHZ}
J.~Herzog, T.~Hibi, X.~Zheng, {\em The monomial ideal of a finite
meet-semilattice,}  Trans. Amer. Math. Soc.  358  (2006),  no. 9, 4119--4134.

\bibitem{H1} T.~Hibi, {\em
  Distributive lattices, affine semigroup rings and algebras with straightening laws,}
  in "Commutative Algebra and Combinatorics, Advanced Studies in Pure Math."
  (M. Nagata and H. Matsumura Eds.) Vol. 11, North-Holland, Amesterdam, 1987, 93--109.

\bibitem{K2} M.~Kummini, {\em Regularity, Depth and Arithmetic Rank of Bipartite Edge Ideals},  preprint (2008).

\bibitem{V}  R.~H.~Villarreal, {\em Unmixed bipartite graphs,}  To appear in Rev. Colombiana Mat.

\bibitem{Z} X.~Zheng, {\em Resolutions of facet ideals,}  Comm. Algebra  32  (2004),  no. 6, 2301--2324.

\end{thebibliography}
\end{document}